\documentclass[reqno,12pt]{amsart}

\usepackage{amscd,amsfonts,mathrsfs,amsthm,amssymb,amsmath}
\usepackage{enumerate}	
\usepackage{ae}
\usepackage{accents}
\usepackage[e]{esvect}
\usepackage[alphabetic, sorted-cites]{amsrefs}
\usepackage[normalem]{ulem}
\usepackage{graphicx}
\usepackage{subcaption}

\DeclareMathOperator{\Graph}{Graph}
\def\H{\bf {H}}
\def\v{\operatorname{v}}
\def\real     #1{{\mathbb R^{#1}}}
\def\natural  #1{{\mathbb N^{#1}}}

\newtheorem{theorem}{Theorem}[section]
\newtheorem{lemma}[theorem]{Lemma}

\newtheorem{corollary}[theorem]{Corollary}

\newtheorem{definition}[theorem]{Definition}
\theoremstyle{definition}
\newtheorem{remark}[theorem]{Remark}

\numberwithin{equation}{section}

\newcommand{\grc}{\mathcal{G}} 
\newcommand{\tp}{\mathcal{P}} 
\newcommand{\tc}{\mathcal{W}} 

\begin{document}

\title[Translating solitons]{Some results on translating solitons of the mean curvature flow}

\author[J. P\'erez-Garc\'ia]{\textsc{J. P\'erez-Garc\'ia}}

\address{Jes\'us P\'erez-Garc\'ia\newline
Departamento de Geometr\'ia y Topolog\'ia\newline
Universidad de Granada\newline
Granada, Spain\newline
{\bf jpgarcia@ugr.es}
}

\date{}
\thanks{J. P\'erez-Garc\'ia is supported by Ministerio de Econom\'ia y Competitividad (FPI grant, BES-2012-055302) and by MICINN-FEDER grant no. MTM2014-52368.}

\begin{abstract}
In this article we prove two non-existence results for translating solitons of the mean curvature flow (\emph{translators} for short) in $\real{m+1}$. We also obtain an upper bound to the maximum height that a compact embedded translator in $\real{3}$ can achieve. On the other hand, we study graphical perturbations of translators, showing that asymptotic graphical perturbations of a graph translator of revolution remain a hypersurface of revolution. Finally, we prove that compact translators that lie between two parallel planes inherit the symmetries of their boundaries curves.
\end{abstract}

\maketitle
\section{Introduction and notation}\label{sec:introduction}
An oriented smooth hypersurface $f:M^m\to\real{m+1}$ is called \emph{translating soliton} of the mean curvature flow (\emph{translator} for short) if its mean curvature vector field $\H$ satisfies the differential equation
\begin{equation*}
\H=\v^\perp,
\end{equation*}
where $\v\in\real{m+1}$ is a fixed unit vector and $\v^\perp$ stands for the orthogonal projection of $\v$ to the normal bundle of the immersion $f$. Translators are important in the singularity theory of the mean curvature flow since they often occur as Type-II singularities. The classic examples of translators (see figure \ref{fig:examples}) are
\begin{itemize}
\item Any hyperplane containing the direction of translation $\v$;
\item The \emph{canonical grim reaper cylinder} $\grc$, which is the product of the grim reaper curve and $\real{m-1}$. A parametrization of the grim reaper curve is given by 
$$\gamma: (-\pi/2,\pi/2)\to \real{2},\quad \gamma(t)=(t,-\log \cos t);$$
\item The \emph{translating paraboloid} or \emph{bowl solution} $\tp$, which is an entire rotationally symmetric strictly convex graphical translator \cite[Lemma 2.2]{clutterbuck2007};
\item A \emph{translating catenoid} or \emph{winglike translator} $\tc=\tc_R$, where $R>0$, which is a non-convex rotationally symmetric graphical translator \cite[Lemma 2.3]{clutterbuck2007}.
\end{itemize}

\begin{figure}[h]
\begin{subfigure}{.45\textwidth}
  \centering
  \includegraphics[width=.45\linewidth]{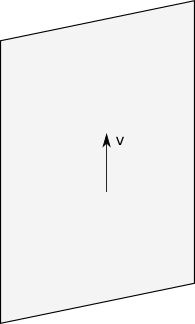}
  \caption{A plane tangential to $\v$}
  \label{fig:sfig1}
\end{subfigure}%
\begin{subfigure}{.45\textwidth}
  \centering
  \includegraphics[width=.63\linewidth]{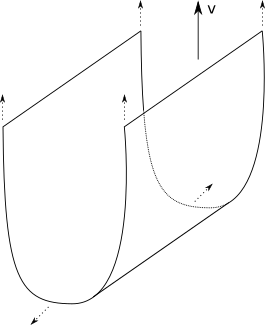}
  \caption{A grim reaper cylinder}
  \label{fig:sfig2}
\end{subfigure}
\par\bigskip 
\begin{subfigure}{1.0\textwidth}
  \centering
  \includegraphics[width=.8\linewidth]{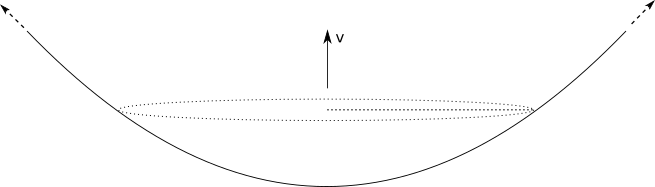}
  \caption{A translating paraboloid or bowl soliton}
  \label{fig:sfig3}
\end{subfigure}
\par\bigskip 
\begin{subfigure}{1.0\textwidth}
  \centering
  \includegraphics[width=.8\linewidth]{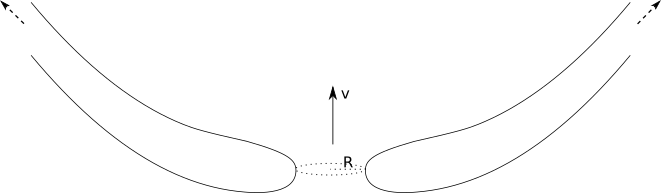}
  \caption{A translating catenoid or winglike translator}
  \label{fig:sfig4}
\end{subfigure}
\caption{Classic examples of translators}
\label{fig:examples}
\end{figure}

For more examples, we refer the reader to \cite[2.2 Examples]{fra14}. On the other hand, we will often discuss the hypothesis in our results using pieces of these examples, in which case the following notation will be very useful: for any $a\in \real{}$, we denote the corresponding closed upper and lower half-space in $\real{m+1}$, respectively, by
\begin{align*}
Z_a^+&=\{ (x_1,\ldots,x_{m+1})\in \real{m+1}: x_{m+1}\geq a \},\\
Z_a^-&=\{ (x_1,\ldots,x_{m+1})\in \real{m+1}: x_{m+1}\leq a \}.
\end{align*}

The aim of this paper is to use these classic examples of translators and the tangency principle (see section \ref{sec:non-existence}) to deduce interesting consequences on translating solitons of the mean curvature flow.

The structure of the paper is as follows. In section \ref{sec:non-existence}, we use the tangency principle to derive two non-existence results for translators. In section \ref{sec:height_estimate} we provide a height estimate for compact translators. In section \ref{sec:graphical_perturbations}, it is shown that a graphical perturbation of a graph translator of revolution $M$ which is asymptotic to $M$, remains a hypersurface of revolution. As an immediate consequence, we give an alternative proof of the uniqueness theorem for complete embedded translating solitons with a single end that are asymptotic to a transla\-ting parabo\-loid \cite[Theorem A]{fra14}. Finally, in section \ref{sec:symmetries}, using the Alexandrov's reflection principle we prove that if a compact translator lies between two parallel planes $P_1$ and $P_2$ which are orthogonal to $\v$, and its bounda\-ry consists of two strictly convex curves contained respectively in $P_1$ and $P_2$, then the translator inherits the symmetries of its boundary.

\section{Non-existence of translators}\label{sec:non-existence}
We begin with the statement of our main tool throughout this paper, the tangency principle.

\begin{theorem}[\textbf{Tangency principle}]\label{thm:tangency_principle}
Let $\Sigma_1$ and $\Sigma_2$ be embedded connected translators in $\real{m+1}$ with boundaries $\partial \Sigma_1$ and $\partial \Sigma_2$.
\begin{enumerate}
\item [\rm (a)] $(${\bf Interior principle}$)$ Suppose that there exists a common point $x$ in the interior of $\Sigma_1$ and $\Sigma_2$ where the corresponding tangent spaces coincide and such that $\Sigma_{1}$ lies at one side of $\Sigma_{2}$. Then $\Sigma_1$ coincides with $\Sigma_2$.
\smallskip
\item [\rm (b)] $(${\bf Boundary principle}$)$
Suppose that the boundaries $\partial \Sigma_1$ and $\partial \Sigma_2$ lie in the same hyperplane $\Pi$ and that the intersection of $\Sigma_{1}$, $\Sigma_{2}$ with $\Pi$ is transversal. Assume that $\Sigma_{1}$ lies at one side of $\Sigma_{2}$ and that there exists a common point of $\partial \Sigma_1$ and $\partial \Sigma_2$ where the surfaces $\Sigma_1$ and $\Sigma_2$ have the same tangent space. Then $\Sigma_1$ coincides with $\Sigma_2$.
\end{enumerate}
\end{theorem}
\smallskip
Roughly speaking, this maximum principle says that two different translators cannot ``touch'' each other at one interior or boundary point. Thanks to the fact that translating solitons are minimal hypersurfaces in a conformally changed Riemannian metric \cite{ilmanen}, the proof is based on the well-known tangency principle for minimal hypersurfaces. For more details, please see \cite[Theorem 2.1]{fra14}.

Let us prove now our first non-existence result about translators.

\smallskip
\begin{theorem}\label{thm:non_existence_I}
Let $f:M^{m} \to \real{m+1}$ be a non-compact embedded connected translator with compact boundary (possiby empty). Then $M$ cannot be contained in any cylinder.
\end{theorem}

\begin{proof}
We argue by contradiction. Suppose that $M\equiv f(M^{m})$ is contained in a cylinder $\mathcal{C}_{r_0}$. We distinguish two cases:

\textbf{Case 1}: The axis of $\mathcal{C}_{r_0}$ is parallel to the direction of translation $\v$.\par
Consider first a \emph{winglike} translator $\tc_{R_{0}}$ with center in the axis of $\mathcal{C}_{r_0}$ and with radius $R_{0}>r_{0}$, so that, in particular, $\tc_{R_{0}}\cap M=\emptyset$. Consider next the family of winglike translators $\{\tc_{R}\}_{0<R\leq R_{0}}$. Since $\tc_{R_{0}}\cap M=\emptyset$, there must be a $R_{1}\in (0, R_{0}]$ such that $\tc_{R_{1}}$ intersects $M$ for the first time. Without loss of generality, we can assume that this first point of contact is an interior point of both surfaces, otherwise it is at the boundary of $M$, in which case it is sufficient to consider the initial winglike translator $\tc_{R_{0}}$ located at a higher height (recall that the boundary of $M$ is compact by hypothesis). Therefore, by the interior tangency principle, $M\subset \tc_{R_{1}}$, which contradicts that $M$ is a non-compact surface contained in $\mathcal{C}_{r_0}$.

\textbf{Case 2}: The axis of $\mathcal{C}_{r_0}$ is not parallel to $\v$.\par
In this case the argument is similar but comparing with a grim reaper cylinder. Let us see it in detail. Due to the compactness of the bounda\-ry and the non-compactness of the translator, there exists a real number $a$ such that $\overline{S}\cap \partial M = \emptyset$, where $S:=(-\pi+a,\pi+a)\times \real{m}$. Let $\hat{\grc}$ be the canonical grim reaper cylinder located in this slab $S$ at a large height so that it does not intersect $M$. Then translate it down until it ``touches'' $M$ for the first time. Observe that this procedure is feasible because $\overline{S}\cap M$ is compact, since by hypothesis the cylinder is tilted. Moreover, as $\overline{S}\cap \partial M = \emptyset$, this point of contact must be in the interior of $M$. Hence, by the interior tangency principle, $M\subset \hat{\grc}$, a contradiction.
\end{proof}
\medskip

\begin{remark}
Let us make here some remarks concerning the previous Theorem \ref{thm:non_existence_I}.
\begin{enumerate}[a)]
\item The result is not true if the translator (with boundary) is compact. A counterexample is the piece of translating paraboloid $\tp$ obtained by cutting this surface with a horizontal plane at any arbitrary but fixed height $a>0$ and considering the lower part, that is, $\tp\cap Z_a^-$.\\
\item The compactness of the boundary is also necessary. A counterexample is the intersection of the canonical grim reaper cylinder $\grc$ with a cylinder of arbitrary but fixed radius $R>0$ and axis the $x_2$-axis; this surface is contained, for instance, in the cylinder of radius $2R$ and axis the $x_2$-axis.
\end{enumerate}
\end{remark}
\medskip

In the following result we prove that there are no translators \emph{that resemble a handle} (see figure \ref{fig:Thm2}). More precisely,

\begin{figure}[h]
\centering
        \includegraphics[width=0.75\linewidth]{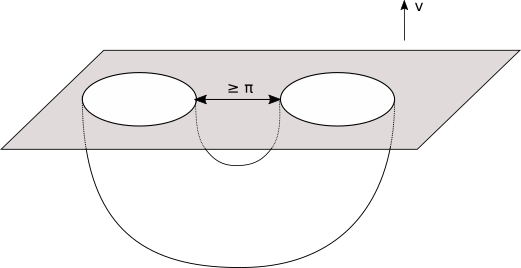}
    \caption{A surface under the conditions of Theorem \ref{thm:non_existence_II}}
    \label{fig:Thm2}
\end{figure}

\begin{theorem}\label{thm:non_existence_II}
There do not exist a connected compact embedded translator in $\real{m+1}$ whose boundary is contained in a hyperplane orthogonal to the direction of translation $\v$ and consists of two strictly convex Jordan curves located at distance greater or equal than $\pi$ and such that one of them is not contained in the region enclosed by the other one.
\end{theorem}

\begin{proof}
We will denote by $f:M\to \real{m+1}$ to an embedding of $M$, and by $P$ to the hyperplane that contained the boundary of $M\equiv f(M)$.

First, note that $M$ must be below the plane $P$. Otherwise, by compactness of $M$, the height function of $M$, $u:=\langle f,\v \rangle$, would attain a global maximum. But recall \cite[Lemma 2.1 (d)]{fra14} that this height function satisfies the equation $\Delta u +|\nabla u|^2=1$, so $u$ does not admit any local maxima in the interior, a contradiction.

Now consider the segment $s$ realizing the distance between the two boundary curves of $M$. The length of this segment is greater or equal than $\pi$ by hypothesis. Let $l$ be the straight line in the direction of $\v$ passing through the middle point of the segment $s$. Place a canonical grim reaper cylinder $\hat{\grc}$ in such a way that its lower generatix coincides with $l$. Observe that $\hat{\grc}$ is strictly contained in a slab $S$ defined as the cartesian product of the segment $s$ times the line $l$. Initially $\hat{\grc}$ does not intersect $M$ because $M$ is below the hyperplane $P$. Then translate $\hat{\grc}$ down following the direction of translation $\v$ until it intersects $M$ for the first time, which necessarily occurs in an interior point of $M$ because any of these translations of $\hat{\grc}$ is strictly contained in the slab $S$. Then, by the interior tangency principle, $M\subset \hat{\grc}$, which is absurd.
\end{proof}
\medskip

\begin{remark}\leavevmode
\begin{enumerate}[a)]
\item In Theorem \ref{thm:non_existence_II}, it is necessary that the boundary curves lie in the same hyperplane. Otherwise the result is not true, as the following example shows: the piece of the translating paraboloid which is between two horizontal hyperplanes: $Z_{a}^{+}\cap \tp\cap Z_{b}^{-}$, where $0<a<b$.\\
\item Moreover, if it is allowed that one of the boundary curves is in the region enclosed by the other one, then there exist translators under the rest of the hypothesis of the theorem \ref{thm:non_existence_II}. For instance, the intersection of a winglike translator with a lower half-space, $\mathcal{W_{R}}\cap Z_a^-$, where, obviously, $a$ is large enough so that this intersection is non-empty.
\end{enumerate}
\end{remark}

\section{A height estimate}\label{sec:height_estimate}
Our aim in this section is to develop a geometric argument for obtaining an upper bound to the maximum height that a compact embedded translator in $\real{3}$ can achieve.

\begin{theorem}\label{thm:height_estimate}
Let $M\subset \real{3}$ be a connected compact embedded translator whose boundary is a connected curve $\Gamma$ contained in a plane $P$ orthogonal to $\v$. Assume that the diameter of $\Gamma$ is $d>0$. Then, for all $p\in M$, the distance in $\real{3}$ from $p$ to $P$ is less or equal than
$$
\begin{cases} 
      -\log \cos \left( \frac{d}{2} \right) & 0<d<\pi \\
      \displaystyle \min_{1<s\leq s_0} C(s)   & d\geq \pi
\end{cases}
$$
where $C:(1,+\infty)\to (0,+\infty)$ is the function given by
$$C(s):=-\left( \frac{d}{\pi}s \right)^2 \log \cos \left( \frac{\pi/2}{s} \right)+\frac{d}{2}\sqrt{\left( \frac{d}{\pi}s \right)^2-1},$$
and 
$$s_0:=\frac{\pi}{2}\frac{1}{\arctan \left( \frac{4-\sqrt{2}}{2} \right)}\approx 1.722.$$
\end{theorem}

\begin{proof}
The idea is to compare $M$ with an appropiate grim reaper cylinder.

First suppose that $0<d<\pi$. Without loss of generality, assume that the diameter of lenght $d$ coincides with
$$\{(x,y,z)\in \real{3}: -d/2\leq x \leq d/2, y=0, z=z_0 \},$$
for an arbitrary but fixed $z_0\in \real{}$. Consider a canonical grim reaper cylinder $\grc$ and observe that, since $d<\pi$, the region between the two parallel planes asymptotic to $\grc$ contains $\Gamma$. Hence, this grim reaper cylinder can be translated down until it does not intersect $M$. Now traslate it up until their first point of contact occurs. By the tangency principle, this must happen at a boundary point of $M$. Observe also that for any point $(x_0,y_0,-\log \cos x_0)$ of the grim reaper cylinder, the width between its two ``wings'' is precisely $2x_0$, so the width is $d$ when the height is $-\log \cos \left( \frac{d}{2} \right)$. In conclusion, this argument shows that $M$ must be contained in the compact region enclosed by the intersection of the horizontal plane $P$ and a canonical grim reaper cylinder whose lowest point is at distance $-\log \cos \left( \frac{d}{2} \right)$ from $P$, which proves the boundedness if $0<d<\pi$.\par

Second, suppose that $d\geq \pi$. In this case a canonical grim reaper cylinder $\grc$ cannot contained $\Gamma$. It is necessary a dilation of factor $\lambda>\frac{d}{\pi}>1$. But then the velocity changes, which does not allow us to use the tangency principle anymore. To overcome this problem, a rotation that maps again the new velocity to $\v$ can be applied to $\lambda\grc$. That is, with an appropiate dilation and rotation we can proceed similarly as above.\\
In detail, consider the canonical grim reaper cylinder
$$\grc=\{(x,y,-\log \cos x): (x,y)\in (-\pi/2,\pi/2)\times \real{} \}.$$
Apply to $\grc$ a dilation of factor $\lambda >1$ such that $\lambda \pi>d$,
$$\lambda\grc=\{\lambda(x,y,-\log \cos x): (x,y)\in (-\pi/2,\pi/2)\times \real{}\},$$
so that $\Gamma$ fits in the slab determined by the dilated grim reaper cylinder $\lambda\grc$. Observe that there are infinite factors of dilation with this property. A way to parametrize them is to consider $\lambda(s):= \frac{d}{\pi}s$, where $s>1$. For brevity, we will usually omit the parameter $s$.\\
Note that with this dilation the translating velocity changes from $\v=(0,0,1)$ to $(1/\lambda)\v=(0,0,1/\lambda)$. A unitary velocity can be achieved noticing that the grim reaper cylinder $\lambda\grc$ is invariant under translations in the directions of $e_2$. Hence, $\lambda\grc$ can be considered as a translator in the direction of $(1/\lambda)\v+ae_2=(0,a,1/\lambda)$, for any $a\in \real{}$. In particular, for $a_0=\sqrt{1-(1/\lambda)^2}$ we have that $\tilde{\v}:=(0,\sqrt{1-(1/\lambda)^2},1/\lambda)$ is a unit vector. Finally, a rotation around the $x$-axis is performanced in order to transform $\tilde{\v}$ into $\v$, so that $M$ and this dilated and tilted grim reaper cylinder can be compared each other, that is, so that they have the same tranlating velocity vector. Specifically, the angle $\alpha$ of rotation must be
$$\frac{1}{\lambda}=\langle \v, \tilde{\v} \rangle= |\v||\tilde{\v}|\cos \alpha = \cos \alpha \Rightarrow \alpha=\arccos \left( \frac{1}{\lambda} \right).$$
Hence, the rotation matrix $R_{x}(\alpha)$ is
$$
R_{x}(\alpha):=
\begin{pmatrix} 
1 & 0 & 0 \\ 
0 & \cos \alpha & -\sin \alpha \\ 
0 & \sin \alpha & \cos \alpha  
\end{pmatrix}=
\begin{pmatrix} 
1 & 0 & 0 \\ 
0 & 1/\lambda & -a_0 \\ 
0 & a_0 & 1/\lambda  
\end{pmatrix}
$$

Therefore, the surface that will be used to compare with $M$ is
\begin{equation}\label{eq:tilted_grim_plane}
\begin{aligned}
R_{x}(\alpha)(\lambda\grc)=&\left\lbrace \left(\lambda x,y+\sqrt{\lambda^2-1}\log \cos x,\sqrt{\lambda^2-1}y-\log \cos x \right):\right.\\
&\quad \left. (x,y)\in (-\pi/2,\pi/2)\times \real{}\right\rbrace.
\end{aligned}
\end{equation}

For brevity, we will denote $R_{x}(\alpha)(\lambda\grc)$ by $\grc_{\lambda,\alpha}$.\\
Now the idea is to translate $\grc_{\lambda,\alpha}$ until it does not intersect $M$ and translate it back until they intersect each other for the first time. By the tangency principle, the first point of contact must be at the boundary of $M$. To make the computations it is convenient to consider the following static situation, which is equivalent: to determine the intersection of $\grc_{\lambda,\alpha}$ with the cylinder $\mathcal{C}$ of diameter $d$,
\begin{equation}\label{eq:cylinder}
\mathcal{C}=\mathcal{C}_{d/2}:=\left\lbrace (x,y,z)\in \real{3}: x^2+y^2= \left(\frac{d}{2}\right) ^2 \right\rbrace,
\end{equation}
and compute the global minimum and maximun of the third coordinate function of the parametrization of this intersection.\\
Combining \eqref{eq:tilted_grim_plane} and \eqref{eq:cylinder}, we obtain that a parametrization of the intersection of $\grc_{\lambda,\alpha}$ and $\mathcal{C}$ is $\gamma_{\pm}:\left[ -\frac{d/2}{\lambda}, \frac{d/2}{\lambda}\right] \to \real{3}$ given by
\begin{align*}
\gamma_{\pm}(x):=\bigg( & \lambda x, \pm \sqrt{\left( d/2 \right)^2-\left( \lambda x \right)^2},\\
& -\lambda^2 \log \cos x \pm \sqrt{\lambda^2-1}\sqrt{\left( d/2 \right)^2-\left( \lambda x \right)^2}  \bigg).
\end{align*} 
The critical points of $\gamma_{\pm}$ correspond to $x=0$:
$$\left( 0,-\frac{d}{2}, -\frac{d}{2}\sqrt{\lambda^2-1} \right)\text{,} \enskip \left( 0,\frac{d}{2}, \frac{d}{2}\sqrt{\lambda^2-1} \right).$$
The points on the boundary of $\grc_{\lambda,\alpha}\cap \mathcal{C}$ are
$$\left( -\frac{d}{2}, 0, -\lambda^2 \log
 \cos \frac{d/2}{\lambda} \right)\text{,} \enskip \left( \frac{d}{2}, 0, -\lambda^2 \log \cos \frac{d/2}{\lambda} \right).$$
Therefore, the global maximum and minimum of the third coordinate function of $\gamma_{\pm}$ are $-\lambda^2 \log \cos \frac{d/2}{\lambda}$ and $-\frac{d}{2}\sqrt{\lambda^2-1}$, respectively. Hence, the boundedness is given in this case by their difference, which is precisely $C(s)$, as claimed.

Now observe that the function $C(s)$ is positive and 
$$\lim_{s\to 1^{+}} C(s)=\lim_{s\to +\infty} C(s)=+\infty,$$
hence $C(s)$ has a global minimum. The problem is that it cannot be computed analytically. Indeed, the critical points of $C(s)$ are the zeros of
$$C'(s)=-2\frac{d^2}{\pi^2}s \log \cos \left( \frac{\pi/2}{s} \right)-\frac{d^2}{2\pi}\tan \left( \frac{\pi/2}{s} \right) +\frac{d^3}{2\pi^2}\frac{s}{\sqrt{\left( \frac{d}{\pi}s \right)^2-1}}.$$
Nevertheless, we can determine an $s_0>1$ such that $C(s)$ is increasing in $(s_0,+\infty)$. Specifically, for all $s>1$,
\begin{align*}
C'(s)&=-2\frac{d^2}{\pi^2}s \log \cos \left( \frac{\pi/2}{s} \right)-\frac{d^2}{2\pi}\tan \left( \frac{\pi/2}{s} \right) +\frac{d^3}{2\pi^2}\frac{s}{\sqrt{\left( \frac{d}{\pi}s \right)^2-1}}\\
&> -\frac{d^2}{2\pi}\tan \left( \frac{\pi/2}{s} \right) +\frac{d^3}{2\pi^2}\frac{s}{\sqrt{\left( \frac{d}{\pi}s \right)^2-1}}\\
&= \frac{d^2}{2\pi}\left(-\tan \left( \frac{\pi/2}{s} \right) +\frac{d}{\pi}\frac{s}{\sqrt{\left( \frac{d}{\pi}s \right)^2-1}}\right)\\
&\geq \frac{d^2}{2\pi}\left(-\tan \left( \frac{\pi/2}{s} \right) +\frac{d}{\pi}\frac{s}{\sqrt{\left( \frac{d}{\pi}s \right)^2}}\right)\\
&= \frac{d^2}{2\pi}\left(-\tan \left( \frac{\pi/2}{s} \right) +1\right).
\end{align*}
Since
$$-\tan \left( \frac{\pi/2}{s} \right) +1\geq 0 \Leftrightarrow s\geq 2,$$
then
$$\min_{s\in (1,+\infty)} C(s)=\min_{1<s\leq 2} C(s).$$

Once we know that $C$ is increasing for $s>2$, we can easily improve the above lower bound of $C'$:
\begin{align*}
C'(s)&= \frac{d^2}{2\pi}\left(\frac{4}{\pi}s \left(-\log \cos \left( \frac{\pi/2}{s} \right)\right) -\tan \left( \frac{\pi/2}{s} \right) +\frac{d}{\pi}\frac{s}{\sqrt{\left( \frac{d}{\pi}s \right)^2-1}}\right)\\
&> \frac{d^2}{2\pi}\left(1\left(1-\cos \left( \frac{\pi/2}{s} \right) \right) -\tan \left( \frac{\pi/2}{s} \right) +1\right)\\
&= \frac{d^2}{2\pi}\left(2-\cos \left( \frac{\pi/2}{s} \right)-\tan \left( \frac{\pi/2}{s} \right) \right).
\end{align*}

Now, observe that
$$2-\cos \left( \frac{\pi/2}{s} \right)-\tan \left( \frac{\pi/2}{s} \right)\geq 0 \Leftrightarrow \cos \left( \frac{\pi/2}{s} \right)+\tan \left( \frac{\pi/2}{s}\right) \leq 2,$$

and taking into account that, due to our previous computations, we can restrict our estimation of $C'$ to the interval $(1,2]$, we have that it is sufficient to find an $s$ such that

$$\cos \left( \frac{\pi/2}{s} \right)+\tan \left( \frac{\pi/2}{s}\right)\leq \cos \left( \frac{\pi}{4} \right)+\tan \left( \frac{\pi/2}{s}\right) \leq 2,$$

that is,

$$ \frac{\sqrt{2}}{2} +\tan \left( \frac{\pi/2}{s}\right) \leq 2 \Leftrightarrow s\geq \frac{\pi}{2}\frac{1}{\arctan \left( \frac{4-\sqrt{2}}{2} \right)},$$

and the proof is complete.

\end{proof}
\medskip

\begin{remark}
The height estimate is valid in a more general setting: in the statement of the Theorem \ref{thm:height_estimate}, instead of considering the diameter $d$ of $\Gamma$, we can assume that the curve $\Gamma$ is strictly contained in a slab of width $d>0$, and the proof remains exactly the same.
\end{remark}

\section{Graphical perturbations of translators}\label{sec:graphical_perturbations}
\begin{definition}[Graphical perturbation]\label{def:graph_perturbation}
Let $N$ be a connected graph hypersurface given by $u:U\subset \real{m}\to \real{m+1}$. Let $M$ be a hypersurface in $\real{m+1}$.\\
We say that $M$ is a graphical perturbation of $N$ if there exists a function $\varphi:U\to\real{}$ such that $M$ can be represented as the graph of $u+\varphi$, that is,
$$M=\Graph{(u+\varphi)}.$$

We will say that $M$ is an asymptotic graphical perturbation of $N$ if $M$ is a graphical perturbation of $N$ and, moreover, for every sequence $\{x_i\}_{i=1}^{+\infty}$ in $U$ such that $\lim\limits_{i\to +\infty}u(x_i)=\infty$ it holds that $\lim\limits_{i\to +\infty}\varphi(x_i)=0$.

Finally, we will say that $M$ is a (an asymptotic) graphical perturbation of $N$ outside a compact set $\mathcal{K}\subset \real{m+1}$ if $M-\mathcal{K}$ is a (an asymptotic) graphical perturbation of $N-\mathcal{K}$.
\end{definition}
\smallskip

\begin{remark}\leavevmode
Note that if there exists a function $\varphi$ as in the previous definition, then it is smooth because it is the difference of two graph hypersurfaces, which are always assumed to be smooth in this paper.
\end{remark}

Roughly speaking, the asymptotic behaviour here means that, outside a bounded region in $M$, $M$ is arbitrarily close to $N$. An example is shown in figure \ref{fig:perturbation}.
\par\bigskip

\begin{figure}[h]
\centering
        \includegraphics[width=0.75\linewidth]{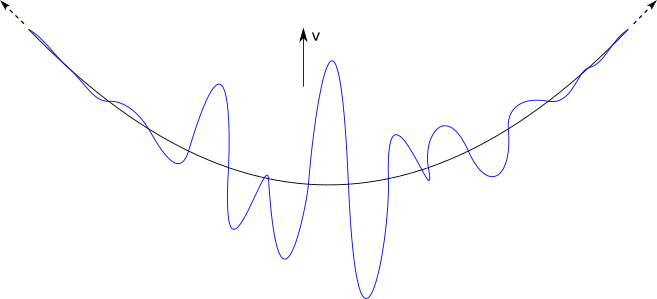}
    \caption{A profile view of an example of an asymptotic graphical perturbation of a translating paraboloid}
    \label{fig:perturbation}
\end{figure}

The goal of this section is to show that if we have two hypersurfaces that are graphically asymptotic outside a compact set, then there are some interesting common properties between them, such as being a graph hypersurface or a hypersurface of revolution. This is the content of our next theorem. Before stating it precisely, let us present the idea of the proof. Basically it is another consequence of the tangency principle, comparing the translator \emph{with a transformation of itself} according to the following scheme:

\begin{enumerate}[1)]
\item Consider $\hat{M}$ a ``copy'' of $M$;
\smallskip
\item Translate $\hat{M}$ up, $\hat{M}\mapsto \hat{M}+t_0v$ for some $t_0>0$, until it does not intersect $M$;
\smallskip
\item Apply an isometry $i:\real{m+1} \to \real{m+1}$ of the ambient space to $\hat{M}+t_0v$ so that $i(\hat{M}+t_0v)\cap M=\emptyset$;
\smallskip
\item Move $i(\hat{M}+t_0v)$ down until it ``touches'' $M$ for the first time.
\end{enumerate}

Then, by the interior tangency principle, $i(\hat{M})= M$. Hence, $M$ is invariant under the isometry $i$. 

We will follow this scheme in the next results.\par

For example,
\begin{itemize}
\item To show that $M$ is a graph hypersurface, take $i=\text{identity}$;
\item To show that $M$ is a hypersurface of revolution, consider as $i$ any arbitrary but fixed rotation around the axis of symmetry.
\end{itemize}

Here by a hypersurface (or set, in general) of revolution in Euclidean space $\real{m+1}$ we mean a hypersurface (set) of $\real{m+1}$ which is invariant by the action of $\text{SO}_{l}(m+1)$, the subgroup of the special orthogonal group $\text{SO}(m+1)$ that fixes a given straight line $l$. We will assume that all the sets of revolution that appear together are sets of revolution with respect to the same axis unless otherwise stated.

\begin{lemma}\label{lem:graphical}
Let $N$ be a connected graph translator in $\real{m+1}$. Assume that $M\subset\real{m+1}$ is, outside a compact set $\mathcal{K}\subset \real{m+1}$, a translator which is a graphical perturbation of $N$. Suppose that the boundary of $M$ is graphical (possibly empty). Then $M$ is graphical.
\end{lemma}
\begin{proof}
Obviously, by definition of graphical perturbation, $M-\mathcal{K}$ is graphical. We have to prove that $M\cap\mathcal{K}$ also is. To this end, just apply the scheme presented above, which clearly works because we deal with a compact region. To avoid contact at the boundary of $M\cap\mathcal{K}$ during step 4, we work from the very beginning with a bigger compact set $\overline{B_{r}(0)}\supset \mathcal{K}$ with $r>0$ sufficiently large so that the boundary created intersecting $M$ with $\overline{B_{r}(0)}$ is graphical. Therefore, the contact at the boundary can occur only when $i(\hat{M})$ comes back to its original position, in which case $M$ is graphical, as claimed.
\end{proof}
\medskip

\begin{corollary}\label{cor:outside_compact_set}
If the definitions given in \ref{def:graph_perturbation} hold outside a compact set, they hold everywhere.
\end{corollary}
\begin{proof}
The case of \emph{graphical perturbation} is precisely the content of the previous lemma \ref{lem:graphical}.\par

With respect to \emph{asymptotic graphical perturbation}, observe simply that this definition is indepedent of what happens in compact regions since it deals with the behaviour of the surfaces at infinity.
\end{proof}
\medskip

\begin{theorem}\label{thm:translator_of_revolution}
Let $N$ be a connected graph translator of revolution in $\real{m+1}$. Suppose that $M$ is, outside a compact set $\mathcal{K}\subset \real{m+1}$, a translator of $\real{m+1}$ which is an asymptotic graphical perturbation of $N$. Assume\- that the boundary of $M$ is graphical (possibly empty) and a set of revolution. Then $M$ is a hypersurface of revolution.
\end{theorem}
\begin{remark}\leavevmode
Under the hypothesis of Theorem \ref{thm:translator_of_revolution}, if there exists the boundary of $M$, then it is not necessarily connected. For instance, the intersection of a winglike translator with two different and parallel horizontal planes. But, in general, due to the rotational symmetry hypothesis on the boundary of $M$, each connected component of the boundary of $M$ must be contained in a horizontal hyperplane, and indeed it must be a circumference.
\end{remark}

\begin{proof}
We will show that the previous scheme works for any arbitrary but fix element $i$ of $\text{SO}_{l}(m+1)$, where $l$ is the axis of symmetry. Indeed, $l$ must be parallel to the direction of translation $\v=(0,\ldots,0,1)$, other\-wise $M$ would not be graphical. Let us consider such an isometry ~$i$.\par

First, by corollary \ref{cor:outside_compact_set}, $M$ is an asymptotic graphical perturbation of $N$ everywhere. Then, there exists $d<\infty$ (for instance, $d:=\max_{U}|\varphi|$) such that
\begin{equation*}\label{rel:slab}
|\varphi(x)|<d \, \text{ for all } x\in U.
\end{equation*}
Geometrically, this means that $M$ is contained in a slab $S$ of diameter $d$ centered at $N$:
$$S:=\{(s_1,\ldots,s_{m},s_{m+1})\in U\times\real{}:|s_{m+1}-u(s_1,\ldots,s_{m})|\leq d, x\in N\},$$
and $M\subset S$.\\
Since $N$ is a hypersurface of revolution by hypothesis, then $S$ is a set of revolution.

We can easily argue now that our scheme works:
\begin{itemize}
\item Step 2 is trivially possible to do because $M$ is graphical;
\item Step 3 is achievable because
\begin{center}
$M\subset S$ (by construction of $S$), $\, (\hat{M}+t_{0}v)\cap S=\emptyset$ (by step 2) $\Rightarrow i(\hat{M}+t_{0}v)\cap M\subset i(\hat{M}+t_{0}v)\cap S=i\big((\hat{M}+t_{0}v)\cap S\big)=\emptyset$;
\end{center}
\item Step 4:\\
The first point of contact cannot be at infinity because $\varphi$ tends to zero at infinity.\\
Since the boundary is a graphical set of revolution, the first point of contact cannot be at the boundary unless the hypersurface returns to its original position, in which case $i(\hat{M})=M$, as claimed.
\end{itemize}
\end{proof}

\medskip
\begin{corollary}\cite[Theorem A]{fra14}
Let $f:M^m\to\real{m+1}$ be a complete embedded translating soliton of the mean curvature flow with a single end that is smoothly asymptotic to a translating paraboloid. Then, $M=f(M^m)$ is a translating paraboloid.
\end{corollary}

\begin{proof}
It is a consequence of our previous theorem \ref{thm:translator_of_revolution}, taking $N$ as the translating paraboloid.\\
Observe that the meaning of \emph{smoothly asymptotic} in \cite[Theorem A]{fra14} is that there exists a sufficiently large $r>0$ such that $M-B_r(0)$ can be written as the graph of a function $g$ such that
\begin{align}\label{identity:asymptotic_behaviour}
g(x)=\frac{1}{2}||x||-\frac{1}{2}\log\big(||x||\big)+O\bigg(\frac{1}{||x||}\bigg),
\end{align}
where $||\cdot||$ denotes the usual euclidean norm in $\real{m}$.\\
Now, taking into account that the translating paraboloid is a graph hypersurface for a function $f\in C^{\infty}(\mathbb{R}^{m})$ satisfying the same asymptotic behaviour
\begin{align}\label{identity:asymptotic_behaviour2}
f(x)=\frac{1}{2}||x||-\frac{1}{2}\log\big(||x||\big)+O\bigg(\frac{1}{||x||}\bigg),
\end{align}
then being \emph{smoothly asymptotic} clearly implies being \emph{an asymptotic graphical perturbation}. Indeed, from the relation $g=f+\varphi$ and from \eqref{identity:asymptotic_behaviour} and \eqref{identity:asymptotic_behaviour2}, we deduce that
$$\varphi=g-f=O\bigg(\frac{1}{||x||}\bigg).$$
That is, it is sufficient to consider as $\varphi$ any smooth function such that $\varphi= O\bigg(\frac{1}{||x||}\bigg)$ (as $||x||\to \infty$), i.e., $\varphi(x)\leq \frac{C}{||x||}$ for all $||x||> r$ and for some constant $C\in \mathbb{R}$.
\end{proof}

\section{Compact translators with symmetric boundary}\label{sec:symmetries}
In this section we apply the method of moving planes \cites{alexandrov,schoen,lopez2013} to study compact translators with symmetric boundary.

\begin{theorem}\label{thm:translators_and_symmetries}
Let $M$ be a connected compact embedded translator in $\real{m+1}$ whose boundary consists of two strictly convex curves $\Gamma_1$ and $\Gamma_2$ contained, respectively, in two parallel planes $P_1$ and $P_2$ which are orthogonal to $\v$. Assume that $M$ lies between the two planes $P_1$ and $P_2$, and suppose that the curves $\Gamma_1$ and $\Gamma_2$ are symmetric with respect to a plane $\Pi$ containing the direction of translation $\v$. Then $M$ is symmetric with respect to the plane $\Pi$.
\end{theorem}
\begin{proof}
Without loss of generality, up to a rigid motion, we can assume that
$$P_1=\{(x_1,\ldots,x_{m+1})\in \real{m+1}: x_{m+1}=0\}$$ 
and 
$$\Pi=\{(x_1,\ldots,x_{m+1})\in \real{m+1}: x_1=0\}.$$

We will apply the Alexandrov's method of moving planes (see \cites{alexandrov,schoen}). We will follow the application of this method in \cite[Section 3]{fra14}, including the notation, which we recall briefly:

The family of planes $\{\Pi(t)\}_{t\in\real{}}$ is given by
$$\Pi(t):=\big\{(x_1,\ldots ,x_{m+1})\in\real{m+1}:x_1=t\big\},$$
and given a subset $A$ of $\real{m+1}$, for any $t\in \real{}$ we define the sets
\begin{eqnarray*}
\delta_{t}(A)&:=&\big\{(x_1,\ldots ,x_{m+1})\in A:x_1= t\big\}=A\cap \Pi(t),\\
A_+(t)&:=&\big\{(x_1,\ldots ,x_{m+1})\in A:x_1\ge t\big\},\\
A_-(t)&:=&\big\{(x_1,\ldots ,x_{m+1})\in A:x_1\le t\big\},\\
A^*_+(t)&:=&\big\{(2t-x_1,\ldots ,x_{m+1})\in\real{m+1}:(x_1,\ldots ,x_{m+1})\in A_+(t)\big\},\\
A^*_-(t)&:=&\big\{(2t-x_1,\ldots ,x_{m+1})\in\real{m+1}:(x_1,\ldots ,x_{m+1})\in A_-(t)\big\}.
\end{eqnarray*}

Note that $A^*_+(t)$ and $A^*_-(t)$ are the image of $A_+(t)$ and $A_-(t)$ by the reflection respect to the plane $\Pi(t)$.

Consider the set
$$\mathcal{A} := \{ t\in [0,t_0]: M_{+}(t) \text{ is a graph over } \Pi \text{ and } M_{+}^{\ast}(t)\geq M_{-}(t) \},$$
where $t_0 :=\max \{ t> 0: M\cap \Pi(t)\neq \emptyset \}$ is a positive real number that exists because of the compactness of $M$. Indeed, $q_0:=M\cap \Pi(t_0)$, the first point of contact between $M$ and a vertical plane coming from $+\infty$, must be a boundary point of $M$, otherwise $M$ would coincide with $\Pi(t_0)$ by the interior tangency principle, which is absurd.

Our goal is to prove that $0\in \mathcal{A}$. The proof of this fact will be divided into 3 claims. 

\textbf{Claim 1}. The set $\mathcal{A}-\{t_0\}$ is not empty. Moreover, if $s\in \mathcal{A}$, then $[s,t_0]\in A$.

To show that $\mathcal{A}-\{t_0\}$, we prove that there exists an $\varepsilon>0$ such that $(t_0-\varepsilon,t_0]\subset \mathcal{A}$.\\
First note that $\Gamma_1\cup \Gamma_2$ is a bi-graph over its plane of symmetry $\Pi$ because, by hypothesis, both boundary curves are strictly convex plane curves. Then, in a neighbourhood around $q_0\in \Gamma_1\cup \Gamma_2$, $M$ is a graph over $\Pi$. Otherwise, as $M$ lies between the planes $P_1$ and $P_2$, a neighbourhood of $M$ around $q_0$ would be contain in the plane $P_i$, for some $i\in \{1, 2\}$, that is, $M$ would not be locally around $q_0$ a translator in the direction of $\v$, which is absurd. In other words, since $q_0$ is in $\Gamma_1\cup \Gamma_2$ and it is the first point of contact between $M$ and $\Pi(t_0)$, by continuity this implies that there exists a sufficiently small $\varepsilon>0$ such that $M_{+}(t)$ is a graph over $\Pi(t)$ for every $t\in (t_0-\varepsilon,t_0]$. Moreover, as $M$ is embedded, considering $\varepsilon>0$ smaller if necessary, it holds that $M_{+}^{\ast}(t)\geq M_{-}(t)$ for every $t\in (t_0-\varepsilon,t_0]$.\\
For the second part of the claim, let $\tilde{t}$ be an arbitrary but fixed number in the interval $(s,t_0)$. Our goal is to prove that $\tilde{t}\in \mathcal{A}$. According to the definition of the set $\mathcal{A}$, there are two conditions to be checked, so the proof falls naturally into two parts or steps.

\emph{Step 1}: $M_{+}(\tilde{t})$ is a graph over $\Pi$.\\
As $s\in \mathcal{A}$, then $M_{+}(s)$ is a graph over $\Pi$.
Therefore, $M_{+}(t)$ is a graph over $\Pi$ for every $t\in [s,t_0]$. In particular, $M_{+}(\tilde{t})$ is a graph over $\Pi$.

\emph{Step 2}: $M_{+}^{\ast}(\tilde{t})\geq M_{-}(\tilde{t})$.\\
On the contrary, if $M_{+}^{\ast}(\tilde{t})\ngeq M_{-}(\tilde{t})$, then, by compactness of $M$, there exists a number $t_1\in [\tilde{t},t_0-\varepsilon)$ such that $M_{+}^{\ast}(t_1)-\delta_{t_1}(M)$ and $M_{-}(t_1)-\delta_{t_1}(M)$ intersect for the first time. Furthermore, this first point of contact is an interior point of $M_{+}^{\ast}(t_1)$ and $M_{-}(t_1)$ because the boundary of $M$ consists of two strictly convex plane curves symmetric with respect to $\Pi$. Then $M_{+}^{\ast}(t_1)=M_{-}(t_1)$ by the interior tangency principle. Thus, $\Pi(t_1)\neq \Pi$ would be a plane of symmetry of $M$, hence, in particular, it would be a plane of symmetry of the curves $\Gamma_1$ and $\Gamma_2$, a contradiction.\par

\textbf{Claim 2}. $\mathcal{A}$ is a closed set of the interval $[0,t_0]$.\\
The argument is identical to the one in \cite[Theorem A]{fra14}: it is proved by contradiction, using the sequential characterization of closed sets; first it is assumed that the graphical condition in $\mathcal{A}$ is not satisfied, which contradicts Claim 1; then the graphical condition and the continuity gives the reflection condition in $\mathcal{A}$.

\textbf{Claim 3}. The minimum of the set $\mathcal{A}$ is $0$.\\
We argue by contradiction. Suppose $s_0:=\min \mathcal{A}>0$. We will show that then there exists $\varepsilon_0>0$ such that $s_0-\varepsilon_0 \in \mathcal{A}$, contradicting that $s_0$ is the minimum of $\mathcal{A}$.\\ Again, we divide the proof into two steps.

\emph{Step 1}: There exists $\varepsilon_1 \in (0,s_0)$ such that $M_{+}^{\ast}(s_0-\varepsilon_1)$ is a graph over \nolinebreak $\Pi$.\\
Since $s_0\in \mathcal{A}$, $M_{+}(s_0)$ is a graph over $\Pi$. Moreover, there is no point in $M_{+}(s_0)$ with normal vector included in $\Pi$. If there were such a point, let us say that its first coordinate is $\tilde{t}\in [s_0,t_0)$, then by the tangency principle at the boundary, $M_{+}^{\ast}(\tilde{t})=M_{-}(\tilde{t})$, that is, $\Pi(\tilde{t})$ would be a plane of symmetry of $M$. In particular, $\Pi(\tilde{t})$ would be a plane of symmetry of the curves $\Gamma_1$ and $\Gamma_2$, which contradicts that $\Pi\neq \Pi(\tilde{t})$ also is. Thus,
$$\xi\{M_{+}(s_0)\}\cap \Pi = \emptyset.$$
As $M$ is compact, we have that there exists $\varepsilon_1 \in (0,s_0)$ such that
$$\xi\{M_{+}(s_0)\}\cap \Pi = \emptyset \quad \text{for all } t\in [s_0-\varepsilon_1,s_0].$$
From this fact, together with the compactness of $M$, it follows that $M_{+}(t)$ can be represented as a graph over the plane $\Pi$ for every $t\in [s_0-\varepsilon_1,s_0]$. In particular, $M_{+}^{\ast}(s_0-\varepsilon_1)$ is a graph over $\Pi$ and the proof of this step is complete.

\emph{Step 2}: There exists $\varepsilon_0 \in (0,\varepsilon_1)$ such that $M_{+}^{\ast}(s_0-\varepsilon_0)\geq M_{-}(s_0-\varepsilon_0)$.\\
We are going to show that there exists $\varepsilon_0 \in (0,\varepsilon_1]$ such that
$$M_{+}^{\ast}(t)\cap M_{-}(t)=\delta_t(M) \quad \text{for all } t\geq s_0-\varepsilon_0,$$
which in particular implies that $M_{+}^{\ast}(s_0-\varepsilon_0)\geq M_{-}(s_0-\varepsilon_0)$, and this step will be proved.\\
We argue by contradiction. If it were not true, then there would exist an increasing sequence $\{t_n\}_{n\in  \natural{}}$ converging to $s_0$ such that
$$\left(M_{+}^{\ast}(t_n)\cap M_{-}(t_n)\right)-\delta_{t_n}(M)\neq \emptyset.$$
For each natural $n$, denote by $P_n=(p_1^n,p_2^n,p_3^n)$ a point in the above set. At this point, we make two key observations:
\begin{equation}\label{eq:observation1}
\left(M_{+}^{\ast}(t)\cap M_{-}(t)\right)-\delta_t(M)\subset M_{-}(s_0-\varepsilon_1)\quad \text{for all } t\in [s_0-\varepsilon_1,t_0],
\end{equation}
\begin{equation}\label{eq:observation2}
M_{+}^{\ast}(s_0)\cap M_{-}(s_0)=\delta_{s_0}(M).
\medskip
\end{equation}
\eqref{eq:observation1} follows from Step 1, that is, from the fact that $M_{+}^{\ast}(s_0-\varepsilon_1)$ is a graph over $\Pi$ for every $t\in [s_0-\varepsilon_1,t_0]$. Therefore, 
$$\big( \left(M_{+}^{\ast}(t)\cap M_{-}(t)\right)-\delta_t(M) \big)\cap S=\emptyset,$$
where $S:=\{(x_1,\ldots,x_{m+1})\in \real{m+1}: s_0-\varepsilon_1 \leq x_1\leq t_0\}$, simply because, in plain language,

``\emph{the reflection of a graph over a plane $\Pi$ with respect to this plane $\Pi$ is always on the right hand side of the left part of the graph}'', 

where orientation (right and left) is considered with respect to the plane $\Pi$. This is a direct consequence of the definitions, in particular from the meaning of being a graph over a plane.\\
On the other hand, \eqref{eq:observation2} follows from the fact that $s_0\in \mathcal{A}$. Indeed, if $M_{+}^{\ast}(s_0)\cap M_{-}(s_0)$ were a set bigger than $\delta_{s_0}(M)$, then, as $M_{+}^{\ast}(s_0)\geq M_{-}(s_0)$, there would be a first point of contact between $M_{+}^{\ast}(s_0)$ and $M_{-}(s_0)$, which would be in the interior because the boun\-dary of $M$ consists of two strictly convex plane curves symmetric with respect to $\Pi$. Then by the interior tangency principle both surfaces would coincide, hence $\Pi(s_0)$ would be a symmetric plane of $M$, a contradiction.\\
Let us come back to the sequence $\{P_n\}_{n\in \natural{}}$. By the compactness of $M$, we can assume without loss of generality that this sequence converges to a point $P_{\infty}=(p_1^\infty,p_2^\infty,p_3^\infty)\in M$. Indeed, since $t_n\nearrow s_0$, $P_{\infty}\in M_{+}^{\ast}(s_0)\cap M_{-}(s_0)=\delta_{s_0}(M)$, where the last equality is by \eqref{eq:observation2}. On the other hand, from \eqref{eq:observation1} we see that $p_1^n\leq s_0-\varepsilon_1$ for each $n$. Thus, $p_1^\infty\leq s_0-\varepsilon_1<s_0$, which contradicts that $P_{\infty}\in \delta_{s_0}(M)$.
\end{proof}
\medskip

\begin{remark}
The assumption that $M$ must be between the two pa\-rallel planes $P_1$ and $P_2$ cannot be dropped, as the following counterexam\-ple shows. Consider the intersection of a winglike solution $\tc$ with two horizontal parallel planes $P_1$ and $P_2$, so that the lower one, $P_1$, contains the radius of $\tc$. Observe that the intersection of each of these planes with $\tc$ consists of two concentric circles. The counterexample is the piece of $\tc$ between these two planes and whose boundary is the inner circle in $P_1$ and the outer circle in $P_2$.
\end{remark}
\medskip

\begin{corollary}
In the setting of the previous theorem \ref{thm:translators_and_symmetries}, if $\Gamma_1$ and $\Gamma_2$ are concentric circles, then $M$ is a hypersurface of revolution.
\end{corollary}
\begin{proof}
Simply note that, by theorem \ref{thm:translators_and_symmetries}, $M$ is symmetric with respect to any plane $\Pi$ containing the direction of translation $\v$. Therefore, $M$ is a hypersurface of revolution around $\v$.
\end{proof}
\medskip

\begin{corollary}
Theorem \ref{thm:translators_and_symmetries} remains true if the boundary of the translator $M$ is assumed to be only one strictly convex plane curve $\Gamma$.
\end{corollary}
\begin{proof}
Observe that in this case the translator lies below the plane $P$ that contains the curve $\Gamma$ because $M$ is compact and, as a translator, its height function cannot attain a local maximum \cite[Lemma 2.1 (d)]{fra14}. Hence, the same argument using the Alexandrov's method proves the corollary.
\end{proof}
\medskip

\begin{corollary}
Let $M$ be a connected compact embedded translator in $\real{m+1}$ whose boundary $\Gamma$ is a $(m-1)$-sphere contained in a hyperplane $P$ orthogonal to $\v$. Then $M$ is the compact piece of the translating paraboloid whose boundary coincides with $\Gamma$.
\end{corollary}
\begin{proof}
Let $\tp$ be the compact piece of the translating paraboloid whose boundary coincides with $\Gamma$. Place $\tp$ above the plane $P$ so that its vertex lies on the same line as the center of $\Gamma$. Then translate it down until they ``touch'' for the first time. There are two possibilities: either they intersect for the first time in an interior point or they do it in a boundary point. In any case, the interior or boundary tangency principle tells us they coincide.
\end{proof}

\textbf{Acknowledgments}. The author thanks Professor Francisco Mart\'in for his great assistance and guidance in the whole preparation of this paper. The author also thanks Professor Magdalena Rodr\'iguez for comments and discussions that improved the manuscript.

\end{document}